\documentclass[11pt]{extarticle}

\usepackage[left=1in,right=1in,top=1in,bottom=1.3in]{geometry}
\usepackage{amsmath, amssymb, amsthm, latexsym, amsfonts, indentfirst, xcolor, mathtools, manfnt, microtype}
\usepackage[utf8]{inputenc}
\usepackage[english]{babel}

\numberwithin{equation}{section} 

\usepackage{indentfirst}
\usepackage{booktabs} 		
\usepackage{ dsfont }

\usepackage{amsthm}
\theoremstyle{plain}
\newtheorem{theorem}{Theorem}
\newtheorem{proposition}[theorem]{Proposition}
\newtheorem{lemma}[theorem]{Lemma}
\newtheorem{corollary}{Corollary}

\theoremstyle{definition}

\usepackage{enumitem}

\newcommand{\FF}{\mathcal{F}}

\newcommand{\m}{\mathcal}
\newcommand{\HH}{\mathcal{H}}
\usepackage{hyperref}
\newtheorem{conjecture}{Conjecture}

\newenvironment{proofw}{\par
  \pushQED{\qed}%
  \normalfont 
  \trivlist
  \item[]\ignorespaces
}{%
  \popQED\endtrivlist\
}

\title{Matchings in permutations}
\date{}
\author{Eduard Inozemtsev\thanks{Moscow Institute of Physics and Technology;
E-mail: \url{eduard_inozemtsev@bk.ru}}, Dmitrii Kolupaev\thanks{Moscow Institute of Physics and Technology, Technion; E-mail: \url{dmitriik@campus.technion.ac.il}}, Andrey Kupavskii\thanks{Moscow Institute of Physics and Technology, Saint-Petersburg State University, Innopolis University;
E-mail: \url{kupavskii@ya.ru}}}

\begin{document}
\maketitle
\begin{abstract}
    We say that two permutations $[n]\to [n]$ intersect if they map some element $x$ to the same element $y$. A matching in a family of permutations is a collection of pairwise disjoint permutations. In this paper, we study families of permutations with no matchings of size $s$. In particular, we obtain a characterization of the largest $s$-matching-free families and a Hilton--Milner type result. We also obtain results for the families of derangements.
\end{abstract}

\section{Introduction}

We denote by $[n]:=\{1,\dots,n\}$ the standard $n$-element set and by
${[n]\choose k}:=\{F\subset [n]: |F|=k\}$ the set of its $k$-element subsets.
A \emph{family} is a collection of sets. For a family $\FF$, let $\nu(\FF)$
denote its \emph{matching number}, that is, the maximum number of pairwise
disjoint sets in $\FF$. If $\nu(\FF)=1$, then any two sets in $\FF$ intersect;
in this case we say that $\FF$ is \emph{intersecting}.

One of the classical results in extremal combinatorics is the
Erd\H{o}s--Ko--Rado theorem~\cite{EKR}, which states that, for $n\geq 2k$, the
largest intersecting family in ${[n]\choose k}$ has size at most
${n-1\choose k-1}$. For $n>2k$, equality holds only if $\FF$ is of the form
$\mathcal S_x:=\{F\in {[n]\choose k}: x\in F\}$ for some $x\in[n]$. Such
families are called \emph{stars} or \emph{trivial}, and $x$ is called their
\emph{center}. If an intersecting family is not pierced by a single element,
then we call it \emph{non-trivial}.

The \emph{covering number} $\tau(\FF)$ of $\FF$ is the size of the smallest set
$X$ such that $X\cap F\ne\emptyset$ for all $F\in\FF$. The family
$\mathcal S_x$ has covering number $1$. Hilton and Milner~\cite{HM} gave a
tight upper bound on the size of an intersecting family in ${[n]\choose k}$
with covering number at least $2$: for $n>2k$,
$|\FF|\leq {n-1\choose k-1}-{n-k-1\choose k-1}+1$. Note that the condition
$\tau(\FF)\geq 2$ is equivalent to $\FF$ being non-trivial. Results
of this type are called {\it stability} results. For a given extremal problem, they
provide a trade-off between the distance from the extremal examples and the
largest possible size of the family.

In 1968, Erd\H{o}s~\cite{E} proposed a conjecture that generalizes the
Erd\H{o}s--Ko--Rado theorem. It determines the largest size of a family
$\FF\subset {[n]\choose k}$ with $\nu(\FF)<s$. There are two natural candidate
extremal families:
\[
    \mathcal A_1:=\{F\in {[n]\choose k}: [s-1]\cap F\ne\emptyset\},
    \qquad
    \mathcal A_k:={[sk-1]\choose k}.
\]
The conjecture of Erd\H{o}s states the following.

\begin{conjecture}[Erd\H{o}s Matching Conjecture~\cite{EMC}]
  If $n\geq sk$ and a family $\FF\subset {[n]\choose k}$ satisfies
  $\nu(\FF)<s$, then $|\FF|\leq \max\{|\mathcal A_1|,|\mathcal A_k|\}$.
\end{conjecture}

Erd\H{o}s proved the conjecture for $n>n_0(s,k)$. By now, the conjecture has
been confirmed in several important ranges, both when $n$ is very close to
$sk$, where the second family is extremal, and when $n>Csk$, where the first
family is extremal. See~\cite{FK2,KoKu} for the most recent results.

We call the range of parameters in which $\mathcal A_1$ is extremal the
\emph{trivial regime}. In analogy with the Hilton--Milner theorem, one may
expect that a Hilton--Milner type stability result should hold for the
Erd\H{o}s Matching Conjecture in the trivial regime. This is indeed the case,
as was shown by Frankl and Kupavskii~\cite{FK6}. Let us first define the
relevant family:
\[
    \mathcal H
    =
    \{H\in {[n]\choose k}: H\cap [s-2]\ne\emptyset\}
    \cup \{[s,s+k-1]\}
    \cup
    \{H\in {[s-1,n]\choose k}: s-1\in H,\ H\cap [s,s+k-1]\ne\emptyset\}.
\]
The family $\mathcal H$ is the union of $s-2$ stars with centers in $[s-2]$ and
the Hilton--Milner family on $[s-1,n]$. Clearly, $\nu(\mathcal H)<s$ and
$\tau(\mathcal H)=s$ for $n\geq sk$.

\begin{theorem}[Frankl and Kupavskii~\cite{FK6}]
\label{t: Hilton-Milner for EMC}
    Let $n\geq 2sk(1+o(1))$, where $o(1)$ is taken with respect to
    $s\to\infty$. Assume that $\mathcal G\subset {[n]\choose k}$ satisfies
    $\nu(\mathcal G)<s$ and $\tau(\mathcal G)\geq s$. Then
    $|\mathcal G|\leq |\mathcal H|=|\mathcal A_1|-{n-k-s+1\choose k-1}+1$.
\end{theorem}

EKR- and EMC-type problems have also been studied for families other than
subsets of $[n]$: permutations, graphs~\cite{EFF}, partitions~\cite{MeMo,Kup54},
simplicial complexes~\cite{Bor3,Kup55}, and vector spaces~\cite{FG,IK}.

This paper is devoted to EMC-type results for permutations. We denote by
$\Sigma_n$ the family of all permutations $[n]\to[n]$. We say that two
permutations $\sigma,\pi\in\Sigma_n$ \emph{intersect} if there exists
$i\in[n]$ such that $\sigma(i)=\pi(i)$. It is convenient to identify a
permutation $\sigma\in\Sigma_n$ with the $n$-element subset of $[n]^2$ given by
its graph, that is, with the set $\{(i,\sigma(i)): i\in[n]\}$. In this way,
families of permutations can be viewed as subfamilies of ${[n]^2\choose n}$,
and intersection of permutations becomes ordinary intersection of sets.

We introduce the notation used throughout the paper. Let
$\sigma\in\Sigma_n$. We denote by $\Sigma_{n,\sigma}$ and
$\Sigma_{n,\overline{\sigma}}$ the families of permutations that intersect
$\sigma$ and do not intersect $\sigma$, respectively. If $\sigma$ is the
identity permutation, then $\m D_n:=\Sigma_{n,\overline{\sigma}}$ is the family
of derangements. Similarly, for a permutation $\sigma$, we denote by
$\m D_{n,\overline{\sigma}}$ the family of derangements that do not intersect
$\sigma$.

We also use the following standard notation:
\[
    \FF(X):=\{F\setminus X: X\subset F,\ F\in\FF\},\qquad
    \FF[X]:=\{F: X\subset F,\ F\in\FF\},\qquad
    \FF[\m S]:=\bigcup_{A\in\m S}\FF[A].
\]
We use this notation for families of permutations, thinking of permutations as
sets of pairs. For a family $\FF$ of permutations and a pair $(x,y)$, we write
$\FF[(x,y)]$ instead of $\FF[\{(x,y)\}]$; thus $\FF[(x,y)]$ consists of the
permutations in $\FF$ that map $x$ to $y$. Finally, we denote
$d_n:=|\m D_n|$ and $d_{n,1}:=|\m D_n[(x,y)]|$, where $x\neq y$.

Deza and Frankl~\cite{DeF} initiated the study of intersection problems for
permutations. In particular, they showed that if $\FF\subset \Sigma_n$ is
intersecting, then $|\FF|\leq (n-1)!$. Their proof easily generalizes to
families with bounded matching number.

\begin{proposition}
    If $\FF\subset \Sigma_n$ and $\nu(\FF)<s$, then
    $|\FF|\leq (s-1)(n-1)!$.
\end{proposition}

\begin{proof}
    Let $\pi$ be an $n$-cycle and let $G$ be the cyclic group generated by
    $\pi$. The family $\Sigma_n$ is partitioned into $(n-1)!$ left cosets
    $\sigma G$. Any two distinct permutations in the same coset are disjoint.
    Hence each coset contains at most $s-1$ permutations from $\FF$, and the
    result follows.
\end{proof}

This proof is simpler than the proof of the Erd\H{o}s--Ko--Rado theorem.
However, advancing further proved difficult. In particular, the proof above
does not characterize the equality case $|\FF|=(s-1)(n-1)!$. Deza and Frankl
conjectured that, for $s=2$, equality $|\FF|=(n-1)!$ holds if and only if
$\FF$ is a star, that is, if it consists of all permutations containing some
fixed pair $(x,y)$. This was proved independently by Larose and
Malvenuto~\cite{LM} and by Cameron and Ku~\cite{CK}.

Let $\sigma\in\Sigma_n$ be a permutation such that $\sigma(1)\neq 1$, and
consider the family $\m{HM}:=\Sigma_{n,\sigma}[(1,1)]\cup\{\sigma\}$. This is
an intersecting family of permutations; more precisely, it is the permutation
analogue of the family that gives equality in the Hilton--Milner theorem for
$k$-uniform sets. Larose and Malvenuto~\cite{LM} conjectured that any largest
non-trivial intersecting family of permutations must be isomorphic\footnote{By \emph{isomorphic to a family $\mathcal G$} we mean
a family of the form $\pi\mathcal G\rho$, where $\pi,\rho\in\Sigma_n$.} to $\m{HM}$. This
was proved for large $n$ by Ellis~\cite{Ell}. We note that the results
mentioned in the last two paragraphs were proved using algebraic methods such
as Hoffman's spectral bound and representation theory of symmetric groups.

Recently, Kupavskii and Zakharov~\cite{KuZa} introduced the method of spread
approximations, which, in particular, gives a combinatorial way to prove
EKR-type results for permutations. It also opens a route to more general
extremal and stability results. Let us mention some recent applications of the
method to permutations. Wang and Xiao~\cite{WX} proved a maximum diversity
result for intersecting families of permutations. Inozemtsev and
Kupavskii~\cite{InKu} proved a Frankl-type degree result for permutations,
which is a far-reaching generalization of Hilton--Milner type results.
Kupavskii and Noskov~\cite{KN} obtained results concerning sunflower-free
families in a general spread setting.

In this paper, we prove uniqueness and stability results for the analogue of
the Erd\H{o}s Matching Conjecture for permutations. Our first result is a
uniqueness theorem for all permutations.

\begin{theorem} \label{emc}
    Let $n,s$ be integers with $s\geq 2$ and
    $s\leq \frac{n}{2^{17}\log n}$. Let $\FF\subset \Sigma_n$ be a family of
    permutations such that $\nu(\FF)<s$. Then
    $|\FF|\leq (s-1)(n-1)!$. Moreover, in the case of equality, $\FF$ must be
    of the form $\bigcup_{i=1}^{s-1}\Sigma_n[(x,y_i)]$ or
    $\bigcup_{i=1}^{s-1}\Sigma_n[(x_i,y)]$, where the $y_i$'s, respectively the
    $x_i$'s, are distinct.
\end{theorem}

In fact, our approach applies to more general classes of permutations. In
particular, we obtain the following analogous result for derangements.

\begin{theorem} \label{emcder}
    Let $n,s$ be integers satisfying $s\leq \frac{n}{2^{17}\log n}$. Let $\FF\subset \m D_n$ be a family of
    derangements such that $\nu(\FF)<s$. Then $|\FF|\leq (s-1)d_{n,1}$.
    Moreover, in the case of equality, $\FF$ must be of the form
    $\bigcup_{i=1}^{s-1}\m D_n[(x,y_i)]$ or
    $\bigcup_{i=1}^{s-1}\m D_n[(x_i,y)]$, where the $y_i$'s, respectively the
    $x_i$'s, are distinct and all pairs are off-diagonal.
\end{theorem}

We also prove a Hilton--Milner type stability result for families of
permutations, in parallel with Theorem~\ref{t: Hilton-Milner for EMC} for sets. Again, this could be extended to more general classes of
permutations, but we omit this for clarity.

\begin{theorem} \label{hm}
    Let $n,s$ be integers satisfying $s\leq \frac{n}{2^{17}\log n}$. Let
    $\FF\subset \Sigma_n$ be a family of permutations such that
    $\nu(\FF)<s$ and $\tau(\FF)\geq s$. Then
    $|\FF|\leq (s-1)(n-1)!-d_{n,1}+1$. Moreover, in the case of equality,
    $\FF$ must be of the form $\rho\FF_0\pi$ or $\rho\FF_0^{-1}\pi$, where
    $\rho,\pi\in\Sigma_n$ and
    \[
        \FF_0
        =
        \bigcup_{i=2}^{s-1}\Sigma_n[(1,i)]
        \cup
        \Sigma_{n,\sigma}[(1,1)]
        \cup
        \{\sigma\},
    \]
    for some permutation $\sigma$ satisfying $\sigma(1)\notin [s-1]$.
\end{theorem}

The extremal example is the union of the family $\m{HM}$ and $s-2$ full stars,
all of which are pairwise disjoint.

Theorems~\ref{emc} and~\ref{hm} follow the same proof strategy, based on the
spread approximation method. We note that Theorem~\ref{hm} implies
Theorem~\ref{emc}. Nonetheless, we give a separate proof of Theorem~\ref{emc}
because it is simpler, and because essentially the same proof gives
Theorem~\ref{emcder}. After proving Theorem~\ref{emc}, we indicate the changes
needed to prove Theorem~\ref{emcder}. The constant $2^{17}$ in the statements
can certainly be improved. It would be interesting to remove the logarithmic
factor in the denominator.

The rest of the paper is organized as follows. In Section~\ref{sect_spreads},
we give the necessary preliminaries on spread approximations and establish the
spreadness of derangements and double derangements, using certain tools from
the theory of permanents. In Section~\ref{sect_emc}, we prove
Theorem~\ref{emc} and then indicate the changes needed for
Theorem~\ref{emcder}. In Section~\ref{sect_stab}, we prove
Theorem~\ref{hm}.

\section{Spread approximations} \label{sect_spreads}

We start with the notion of spreadness, the spread lemma, and some basic
properties of spread families. We say that a family $\FF \subset {[n] \choose k}$
is $r$-spread, for some $r\geq 1$, if
\[
    \frac{|\FF(X)|}{|\FF|} \le r^{-|X|}
\]
for every $X\subset [n]$. The following statement is a variant due to
Stoeckl~\cite{St} of the breakthrough result of Alweiss, Lovett, Wu and
Zhang~\cite{Alw}.

\begin{theorem}[Spread Lemma, {\cite{Alw,St}}] \label{spreadtheorem}
    If, for some $n,k,r\geq 1$, a family
    $\FF \subset {[n] \choose \leq k}$ is $r$-spread and $W$ is a
    $\beta\delta$-random subset of $[n]$, then
    \[
        \Pr[\exists F\in \FF : F\subset W]
        \geq
        1-\left(\frac{2}{\log_2(r\delta)}\right)^\beta k.
    \]
\end{theorem}

The next lemma states that any large subfamily of a spread family is also
spread.

\begin{lemma} \label{spreadsubfamily}
    Let $\FF$ be an $r$-spread family, and let $\HH\subset \FF$ satisfy
    $|\HH|\geq c|\FF|$ for some $c\in(0,1)$. Then $\HH$ is $cr$-spread.
\end{lemma}

\begin{proof}
    Indeed, for any set $X$ we have
    \[
        \frac{|\HH(X)|}{|\HH|}
        \leq
        \frac{|\FF(X)|}{c|\FF|}
        \le
        \frac{r^{-|X|}}{c}
        \leq
        (cr)^{-|X|}.
    \]
\end{proof}
\begin{lemma}\label{maxrelhom}
    Fix $r\geq 1$ and a family $\FF$. Let $X$ be an inclusion-maximal set
    satisfying
    \[
        |\FF(X)|\geq r^{-|X|}|\FF|.
    \]
    Then $\FF(X)$ is $r$-spread.
\end{lemma}
\begin{proof}
    Let $Y$ be disjoint from $X$. If $Y=\emptyset$, there is nothing to prove.
    Otherwise, by the maximality of $X$,
    \[
        |\FF(X\cup Y)|<r^{-|X|-|Y|}|\FF|\le
        r^{-|Y|}|\FF(X)|.
    \]
    Hence $\FF(X)$ is $r$-spread.
\end{proof}

\begin{lemma} \label{largecontainsspread}
    If $\FF\subset {[n] \choose \leq k}$ and $|\FF|>r^k$, then $\FF(X)$ is
    $r$-spread for some $X$ such that $|\FF(X)|>1$.
\end{lemma}

\begin{proof}
    If $\FF$ is $r$-spread, then we may take $X=\emptyset$. Otherwise, take $X$ as in Lemma~\ref{maxrelhom}. Then $\m F(X)$ is $r$-spread. Moreover, $|\FF(X)|\geq 2$, since otherwise
    \[
        |\FF(X)|=1<r^{-k}|\FF|\leq r^{-|X|}|\FF|.
    \]
    \end{proof}

The following convenient technical lemma was proved by Keevash, Lifshitz, Long
and Minzer~\cite{Kee21}. It is essentially an equivalent formulation of the
rainbow version of the Erd\H{o}s Matching Conjecture.

\begin{lemma}[\cite{Kee21}] \label{lemkee}
    Let $\m H_1,\ldots,\m H_s$ be upwards-closed families in $2^{[m]}$.
    Assume that, for some $p\in(0,1)$ and a $p$-random set $W$, we have
    \[
        \Pr[W\in \m H_i]\geq 3sp
    \]
    for every $i\in[s]$. Then there exist pairwise disjoint sets
    $A_1\in \m H_1,\ldots,A_s\in \m H_s$.
\end{lemma}

Next, we discuss the spreadness of several families of permutations, viewed as
families of sets.

We say that $S\subset [n]^2$ is a \emph{partial permutation} if $S\subset\sigma$
for some permutation $\sigma\in\Sigma_n$. For every partial permutation $S$,
we have  $|\Sigma_n(S)|=(n-|S|)!.$
If $S$ is not a partial permutation, then $|\Sigma_n(S)|=0$. We say that a
family $\FF$ is $(r,q)$-spread if, for every set $A$ of size at most $q$, the
family $\FF(A)$ is $r$-spread. We shall need the following three lemmas.

\begin{lemma} \label{sspread}
    The family $\Sigma_n$ is $\left(\frac n4,\frac n4\right)$-spread.
\end{lemma}

\begin{proofw}
    Let $S$ be a partial permutation with $|S|\leq n/4$. It is enough to show
    that $\Sigma_n(S)$ is $\frac n4$-spread. Let $X\supset S$.

    If $X$ is not a partial permutation, then $\Sigma_n(X)=\emptyset$ and there
    is nothing to prove. Otherwise,
    \[
        \frac{|\Sigma_n(S)|}{|\Sigma_n(X)|}
        =
        \frac{(n-|S|)!}{(n-|X|)!}.
    \]
    Put $m=n-|S|$ and $t=|X|-|S|$. Then
    \[
        \frac{(n-|S|)!}{(n-|X|)!}
        =
        m(m-1)\cdots(m-t+1)
        \geq
        (m!)^{t/m}.
    \]
    Since $m!\geq (m/e)^m$, we get
    \[
        \frac{|\Sigma_n(S)|}{|\Sigma_n(X)|}
        \geq
        \left(\frac{n-|S|}{e}\right)^{|X|-|S|}
        >
        \left(\frac n4\right)^{|X|-|S|},
    \]
    because $|S|\leq n/4$.
\end{proofw}

Next, we shall bound the spreadness of the family of derangements. Let us first discuss some simple properties of that family.  

 Using inclusion-exclusion principle, it is easy to see that $|\m{D}_n| = n!\sum_{i=1}^n \frac{(-1)^{i+1}}{i!}$.  It is a well-known fact that $|\m{D}_n| = \lceil \frac{n!}{e} \rfloor$, i.e. $|\m{D}_n|$ is the nearest integer to $\frac{n!}{e}$. Let us show that  $$d_{n,1} = d_{n-1} + d_{n-2}.$$
Indeed, we can write $\m{D}{[(x, y)]} = \FF_1 \cup\FF_2$, where $\FF_1 = \{ \pi \in \m{D}[(x, y)] \mid \pi(2) = 1 \}$, $\FF_2 = \{ \pi \in \m{D}[(x, y)] \mid \pi(2) \neq 1 \}$. It is easy to see that $|\FF_1| = d_{n-2}$ and $|\FF_2| = d_{n-1}$.

\begin{lemma} \label{dspread}
    The family $\m D_n$ is $\left(\frac n{12},\frac n4\right)$-spread.
\end{lemma}

\begin{proofw}
    Let $S$ be a partial permutation such that $|S|\leq n/4$ and
    $\m D_n(S)\neq\emptyset$. Then
    \[
        |\m D_n(S)|
        \geq
        d_{n-|S|}
        \geq
        \frac{(n-|S|)!}{e}-1
        >
        \frac{(n-|S|)!}{3}
        =
        \frac{|\Sigma_n(S)|}{3}.
    \]
    Since $\m D_n(S)\subset \Sigma_n(S)$ and $\Sigma_n(S)$ is $\frac n4$-spread
    by Lemma~\ref{sspread}, Lemma~\ref{spreadsubfamily} implies that
    $\m D_n(S)$ is $\frac n{12}$-spread.
\end{proofw}

In order to lower bound the spreadness of the family of double derangements, we
need the following consequence of the Egorychev--Falikman theorem~\cite{E,F}.

\begin{lemma} \label{permanent-min-degree}
    Let $A$ be an $N\times N$ $(0,1)$-matrix such that every row and every
    column contains at least $N-2$ ones. Then, for all $N\geq 400$,
    \[
        \operatorname{perm}(A)>\frac{1}{7.5}N!.
    \]
\end{lemma}

\begin{proof}
    Let $H$ be the bipartite graph, with both parts of size $N$, whose edges
    correspond to the zeros of $A$. Thus every vertex of $H$ has degree at most
    $2$.

    We shall add edges to $H$, or equivalently change some additional ones of
    $A$ to zeros, while keeping all degrees in $H$ at most $2$. This can only
    decrease the permanent. Let $H'$ be a maximal graph obtained in this way,
    and let $A'$ be the corresponding matrix. It is enough to prove the desired
    lower bound for $\operatorname{perm}(A')$.

    We first record the structure of $H'$. Let $L_{<2}$ and $R_{<2}$ be the sets
    of vertices of degree less than $2$ in the two parts. By maximality, every
    $u\in L_{<2}$ is adjacent in $H'$ to every $v\in R_{<2}$; otherwise, we
    could add the edge $uv$. Hence 
    $|L_{<2}|,|R_{<2}|\le 1$. Moreover, the total deficiency from degree $2$ is the
    same on the two sides. Therefore, either $H'$ is $2$-regular, or there is
    exactly one vertex $u$ on the left and exactly one vertex $v$ on the right
    of degree $1$, and they are adjacent.

    In the first case, every row and every column of $A'$ contains exactly
    $N-2$ ones. By the Egorychev--Falikman theorem~\cite{E,F}, applied to
    $A'/(N-2)$, we have
    \[
        \operatorname{perm}(A')
        \geq
        (N-2)^N\frac{N!}{N^N}
        =
        \left(1-\frac{2}{N}\right)^N N!
        >
        \frac{1}{7.5}N!,
    \]
    where the last inequality holds, for instance, for all $N\geq 200$.

    It remains to consider the second case. Let $u$ and $v$ be the unique
    vertices of degree $1$ in the two parts. Then $uv$ is an edge of $H'$, so
    the entry $a'_{uv}$ is zero, while all other entries in row $u$ are equal to
    $1$. Let $A'_{uj}$ denote the matrix obtained from $A'$ by deleting row $u$
    and column $j$.

    Since column $v$ has no zeros except in row $u$, for every $j\neq v$ the
    matrix $A'_{uj}$ is obtained from $A'_{uv}$, up to a relabelling of columns,
    by replacing one column with an all-one column. Hence
    \[
        \operatorname{perm}(A'_{uj})
        \geq
        \operatorname{perm}(A'_{uv})
        \qquad\text{for every } j\neq v.
    \]
    Expanding the permanent along row $u$, we obtain
    \[
        \operatorname{perm}(A')
        =
        \sum_{j\neq v}\operatorname{perm}(A'_{uj})
        \geq
        (N-1)\operatorname{perm}(A'_{uv}).
    \]
    The matrix $A'_{uv}$ has size $(N-1)\times(N-1)$, and every row and every
    column contains exactly $N-3$ ones. Applying the Egorychev--Falikman theorem to $A'_{uv}$ and combining with the last displayed inequality  gives
    \[
        \operatorname{perm}(A')
        \geq
        (N-1)(N-3)^{N-1}
        \frac{(N-1)!}{(N-1)^{N-1}}
        =
        \frac{N-1}{N}
        \left(1-\frac{2}{N-1}\right)^{N-1}N!
        >
        \frac{1}{7.5}N!,
    \]
    where the last inequality holds, for instance, for all $N\geq 400$.

    Since $\operatorname{perm}(A)\geq \operatorname{perm}(A')$, the lemma
    follows.
\end{proof}

\begin{lemma} \label{ddspread}
    The family $\m D_{n,\overline{\sigma}}$ is
    $\left(\frac n{30},\frac n4\right)$-spread for $n\geq 600$.
\end{lemma}

\begin{proof}
    By Lemma~\ref{spreadsubfamily}, it is enough to show that, for every
    partial permutation $S$ such that $s:=|S|\leq n/4$ and
    $\m D_{n,\overline{\sigma}}(S)\neq\emptyset$, we have
    \[
        |\m D_{n,\overline{\sigma}}(S)|
        \geq
        \frac{1}{7.5}|\Sigma_n(S)|.
    \]

    Put $N:=n-s$. After fixing the pairs of $S$, we delete the corresponding
    rows and columns. The number $|\m D_{n,\overline{\sigma}}(S)|$ is the
    permanent of an $N\times N$ $(0,1)$-matrix $A$. In each remaining row and
    each remaining column, at most two entries are forbidden: the identity
    entry and the $\sigma$-entry. Hence every row and every column of $A$
    contains at least $N-2$ ones.

    Since $s\leq n/4$ and $n\geq 600$, we have $N\geq 450$. Therefore, by
    Lemma~\ref{permanent-min-degree},
    \[
        |\m D_{n,\overline{\sigma}}(S)|
        =
        \operatorname{perm}(A)
        >
        \frac{1}{7.5}N!
        =
        \frac{1}{7.5}|\Sigma_n(S)|,
    \]
    as required.
\end{proof}

The following theorem is a spread approximation result adapted for families
with no matchings. It says that a family without an $s$-matching can be
approximated by a family of much lower uniformity, which also has no
$s$-matching.

\begin{theorem} \label{spreadapproxtheorem}
    Let $n,k,s,q$ be positive integers. Let $\m A\subset {[n]\choose k}$ and
    let $\FF\subset \m A$. Suppose that $\m A$ is $r$-spread for
    \[
        r>
        \max\{4(s-1)q,\; 2^9 s\log_2(2k)\}
    \]
    and that $\nu(\FF)<s$. Then there exist a family
    $\m S\subset {[n]\choose \leq q}$ with $\nu(\m S)<s$, and a family
    $\FF'\subset \FF$, such that the following hold:
    \begin{enumerate}[label=(\roman*)]
        \item $\FF\setminus \FF'\subset \m A[\m S]$;
        \item for every $B\in \m S$, there is a family $\FF_B\subset \FF$ such
        that $\FF_B(B)$ is $\frac r2$-spread;
        \item $|\FF'|\leq 2^{-q-1}|\m A|$.
    \end{enumerate}
\end{theorem}

\begin{proof}
    We construct $\m S$ by the following greedy procedure. Start with
    $\FF^1:=\FF$. At step $i$, if $\FF^i=\emptyset$, stop. Otherwise, choose an
    inclusion-maximal set $S_i$ such that
    \[
        |\FF^i(S_i)|
        \geq
        (r/2)^{-|S_i|}|\FF^i|.
    \]
    If $|S_i|>q$, stop. Otherwise, put
    \[
        \FF^{i+1}:=\FF^i\setminus \FF^i[S_i].
    \]

    By Lemma~\ref{maxrelhom}, the family $\FF^i(S_i)$ is
    $\frac r2$-spread. Let $m$ be the step at which the procedure stops, and put
    \[
        \m S:=\{S_1,\ldots,S_{m-1}\}.
    \]
    Clearly, $|S_i|\leq q$ for each $i\in[m-1]$. For $B=S_i$, define
    $\FF_B:=\FF^i[S_i]$. Then $\FF_B(B)=\FF^i(S_i)$, and so (ii) holds.

    We now define $\FF'$. If the procedure stops because $\FF^m=\emptyset$, put
    $\FF':=\emptyset$. Otherwise, it stops because $|S_m|>q$, and we put
    $\FF':=\FF^m$. In this case,
    \[
        |\FF^m|
        \leq
        (r/2)^{|S_m|}|\FF^m(S_m)|
        \leq
        (r/2)^{|S_m|}|\m A(S_m)|
        <
        (r/2)^{|S_m|}r^{-|S_m|}|\m A|
        =
        2^{-|S_m|}|\m A|.
    \]
    Since $|S_m|>q$, this gives
    \[
        |\FF'|\leq 2^{-q-1}|\m A|.
    \]
    Moreover, by construction,
    \[
        \FF\setminus \FF'\subset \m A[\m S].
    \]
    Thus (i) and (iii) hold.

    It remains to verify that $\nu(\m S)<s$. Suppose, to the contrary, that
    there are pairwise disjoint sets $T_1,\ldots,T_s\in\m S$. Put
    \[
        Z:=\bigcup_{j=1}^s T_j,
        \qquad
        Z_i:=Z\setminus T_i.
    \]
    Note that $|Z_i|\leq (s-1)q$. Recall that each family
    $\FF_{T_i}(T_i)$ is $\frac r2$-spread.

    Define $\m G_i:=\{A\in \FF_{T_i}(T_i): A\cap Z_i=\emptyset\}$. Then
    \[
        |\m G_i|
        \geq
        |\FF_{T_i}(T_i)|
        -
        \sum_{x\in Z_i}|\FF_{T_i}(T_i\cup\{x\})|.
    \]
    Using the $\frac r2$-spreadness of $\FF_{T_i}(T_i)$, we get
    \[
        |\m G_i|
        \geq
        \left(1-\frac{|Z_i|}{r/2}\right)|\FF_{T_i}(T_i)|
        \geq
        \frac12|\FF_{T_i}(T_i)|,
    \]
    because $r/2\geq 2(s-1)q\geq 2|Z_i|$. Therefore, by
    Lemma~\ref{spreadsubfamily}, each $\m G_i$ is $\frac r4$-spread.

    We now apply Theorem~\ref{spreadtheorem}. Put
    \[
        \beta:=\log_2(2k),
        \qquad
        \delta:=\frac{1}{8s\log_2(2k)}.
    \]
    Then $\beta\delta=\frac{1}{8s}$. Moreover,
    $
        \frac r4\delta\geq 2^4,
    $
    by our assumption $r\geq 2^9s\log_2(2k)$. Hence Theorem~\ref{spreadtheorem}
    implies that a $\frac{1}{8s}$-random subset $W$ of $[n]\setminus Z$
    contains a member of $\m G_i$ with probability at least
    \[
        1-
        \left(\frac{2}{\log_2 2^4}\right)^{\log_2(2k)}k
        =
        1-2^{-\log_2(2k)}k
        =
        \frac12.
    \]

    Let $\m H_i$ be the upward closure of $\m G_i$ in $2^{[n]\setminus Z}$.
    Then the above implies that
    \[
        \Pr[W\in \m H_i]\geq \frac12.
    \]
    Since $p:=1/(8s)$ satisfies $3sp=3/8<1/2$, we may apply
    Lemma~\ref{lemkee}. We obtain pairwise disjoint sets
    $A_1\in\m H_1,\ldots,A_s\in\m H_s$. Hence there are pairwise disjoint sets
    $B_1\in\m G_1,\ldots,B_s\in\m G_s$.
    But then
    $
        T_1\cup B_1,\ldots,T_s\cup B_s
    $
    are pairwise disjoint members of $\FF$, contradicting $\nu(\FF)<s$.
    Therefore $\nu(\m S)<s$, as required.
\end{proof}

The next theorem and corollary can be understood as follows: if a family has a
non-trivial approximation $\m S$, then $\m A[\m S]$ is significantly smaller
than a trivial union of stars.

Here, we say that $\m S$ is \emph{trivial} if it consists of singletons only. A family $\m S$ with $\nu(\m S) < s$ is {\it maximal} if for any  $A\in \m S$ and $B\subsetneq A$, the family $\m S':=\m S\setminus \{A\}\cup \{B\}$ satisfies $\nu(\m S')\ge s$.

\begin{theorem}[\cite{FK}, Theorem~16] \label{nontrivialsimple}
    Let $\varepsilon\in(0,1]$ and let $n,r,q\geq 1$, $s\geq 2$, be integers
    such that
    \[
        \varepsilon r\geq 8e(s-1)q.
    \]
    Let $\m A\subset 2^{[n]}$ be $(r,1)$-spread, and let
    $\m S\subset {[n]\choose \leq q}$ be a maximal non-trivial family with
    $\nu(\m S)<s$. Suppose that $\m S$ contains exactly $l$ singleton members
    $x_1,\ldots,x_l\in[n]$. Then there exists an element $x\in[n]$ such that
    \[
        |\m A[\m S]|
        \leq
        \left|\bigcup_{j=1}^l \m A[x_j]\right|
        +
        \varepsilon(s-1-l)|\m A[x]|.
    \]
\end{theorem}

\begin{corollary} \label{corsimple}
    Let $x\in[n]$ be such that $|\m A[x]|$ is maximal. In the notation of
    Theorem~\ref{nontrivialsimple}, if $\m S$ is non-trivial and
    $\nu(\m S)<s$, then
    \[
        |\m A[\m S]|
        \leq
        (s-2+\varepsilon)|\m A[x]|.
    \]
\end{corollary}

\begin{proof}
    Since $\m S$ is non-trivial, it contains at most $s-2$ singleton members.
    Therefore, in Theorem~\ref{nontrivialsimple} we have $l\leq s-2$. Taking
    $x$ so that $|\m A[x]|$ is maximal, we get
    \[
        |\m A[\m S]|
        \leq
        l|\m A[x]|+\varepsilon(s-1-l)|\m A[x]|
        \leq
        (s-2+\varepsilon)|\m A[x]|.
    \]
\end{proof}

\section{Proof of Theorem~\ref{emc}} \label{sect_emc}

To prove Theorem~\ref{emc} and~\ref{emcder}, we need the following lemma. We state it in a rather general form, but with some prescribed functions for spreadness and the remainder. We note that those could be adjusted depending on the application. We will apply the lemma  for the ambient families being derangements and double derangements.

\begin{lemma} \label{lemc}
    Let $n,s$ be integers satisfying
    $
        s \leq \frac{n}{2^{17}\log n}.
    $
    Let $\m X\subset \Sigma_n$ be a family of permutations that is $(n/100, 1)$-spread. Let $\FF \subset \m X$ be a family that satisfies
    $\nu(\FF)<s$. Let $\m Y\subset \m X$ be the largest subfamily of all permutations containing one of $(s-1)$ fixed elements. Then
    \[
        |\FF| \leq |\m Y|+n^{-4}|\m X|.
    \]
\end{lemma}
Note that this lemma itself gives an approximate variant of Theorems~\ref{emc} and~\ref{emcder}.
\begin{proof}
  
    We apply Theorem~\ref{spreadapproxtheorem} to $\FF$ with
    $q=\lfloor 4\log n\rfloor$. It is easy to see that the requirements on $r$ from Theorem~\ref{spreadapproxtheorem} are satisfied, notably $\frac n{100}>2^9s\log_2(2n)$.  This gives a family
    $\m S\subset {[n]^2 \choose \leq q}$ with $\nu(\m S)<s$ and a subfamily
    $\FF'\subset \FF$ such that
    \[
        \FF\setminus \FF' \subset \m X[\m S]
        \qquad\text{and}\qquad
        |\FF'|\leq n^{-4}|\m X|.
    \]
    It remains to bound $|\m X[\m S]|$.

    Assume first that $\m S$ is trivial. Then $|\m X[\m S]|\le |\m Y|$ by the definition of $\m Y$. Assume $\m S$ is non-trivial and contains $l\le s-2$ singletons. We then apply 
    Theorem~\ref{nontrivialsimple} with $\varepsilon=1/200$ (again, the condition on $r$ is readily checked), and get 
    $$
        |\m F[\m S]|
        \leq
        |\m Y_l|+\frac 1{200}(s-l-1)\max_{v}|\m X[v]|,
    $$
where $\m Y_l\subset X$ is the largest family of permutations from $X$ containing one of some fixed $l$ singletons. Let us show that the right-hand side is at most $|\m Y|$. First, we note that, by $n/100$-spreadness, $\max_{v}|\m X[v]|\le \frac{100}n|\m X|,$ and thus the second term in the summation is upper bounded by $\frac 1{2n} (s-l-1)|\m X|$. Next, let us lower bound $|\m Y_{l+1}|-|\m Y_l|$. By spreadness, we have 
$$
|\m X\setminus \m Y_l|\ge \Big(1-\frac{s}{n/100}\Big)|\m X|\ge 0.99 |\m X|,
$$
where the second inequality is in our assumption on $s$. Thus, via averaging, we can find a singleton $v$, such that $|(\m X\setminus \m Y_l)(v)|\ge \frac 1n |\m X\setminus \m Y_l|\ge \frac{0.99}n |\m X|.$ Adding all permutations containing $v$ to $\m Y_l$, we get that $|\m Y_{l+1}|-|\m Y_l|\ge \frac{0.99}n |\m X|.$ Consequently, $|\m Y|-|\m Y_l|\ge \frac{0.99}n(s-l-1)|\m X|,$ and we conclude that in the case when $\m S$ is non-trivial we have 
$$|\m F[\m S]|< |\m Y|.$$
    Finally, in both cases
    \[
        |\FF|
        \leq
        |\m Y|+|\FF'|
        \leq
        |\m Y|+n^{-4}|\m X|.
    \]
\end{proof}

We are now ready to prove Theorem~\ref{emc}.

\begin{proofw}
    \underline{\emph{Proof of Theorem~\ref{emc}}}.

Take a family $\m F$ of size $(s-1)(n-1)!$.
    Recall that the ambient family $\Sigma_n$ is
    $\left(\frac n4,\frac n4\right)$-spread by Lemma~\ref{sspread}.
    We apply Theorem~\ref{spreadapproxtheorem} to $\FF$ with
    $q=\lfloor 4\log n\rfloor$. It is easy to see that the requirements on $r$ from Theorem~\ref{spreadapproxtheorem} are satisfied, notably $\frac n{4}>2^9s\log_2(2n)$.  This gives a family
    $\m S$ with $\nu(\m S)<s$ and a subfamily $\FF'\subset \FF$ such that
    \[
        \FF\setminus \FF' \subset \Sigma_n[\m S]
        \qquad\text{and}\qquad
        |\FF'|\leq \frac{1}{n^4}n!.
    \]

    Suppose first that $\m S$ is non-trivial. Applying
    Corollary~\ref{corsimple} with $\varepsilon=\frac12$, we obtain
    \[
        |\Sigma_n[\m S]|
        \leq
        \left(s-\frac32\right)(n-1)!.
    \]
    Hence
    \[
        |\FF|
        \leq
        |\Sigma_n[\m S]|+|\FF'|
        \leq
        \left(s-\frac74\right)(n-1)!+\frac{1}{n^4}n!
        <
        (s-1)(n-1)!,
    \]
    a contradiction.

    Thus $\m S$ is trivial. Hence there exist pairs
    $
        (x_1,y_1),\ldots,(x_t,y_t)\in [n]^2,$ $
        t\leq s-1,
    $
    such that
    \[
        \Sigma_n[\m S]
        \subset
        \bigcup_{i=1}^{t}\Sigma_n[(x_i,y_i)]:=\m U.
    \]
    We claim that the stars in this union are pairwise disjoint. Indeed, if
    \[
        \Sigma_n[(x_i,y_i)]\cap \Sigma_n[(x_j,y_j)]\neq\emptyset
    \]
    for some $i\neq j$, then this intersection has size at least $(n-2)!$.
    Therefore
    \[
        |\m U|
        \leq
        (s-1)(n-1)!-(n-2)!.
    \]
    Since $|\FF'|\leq n!/n^4 < (n-2)!$, we get
    \[
        |\FF|
        \leq
        |\m U|+|\FF'|
        <
        (s-1)(n-1)!,
    \]
    a contradiction. Similarly, we must have $t=s-1.$    

    We now show that $\FF\subset \m U$. Suppose, to the contrary, that there
    exists a permutation $\sigma\in\FF\setminus \m U$. Equivalently,
    $
        \sigma(x_i)\neq y_i$ $
        \text{for all } i\in[s-1].
    $
    Let $\m A$ be the subfamily of $\FF\setminus \FF'$ consisting of all
    permutations that do not intersect $\sigma$. Then
    $
        \m A\subset \Sigma_{n,\overline{\sigma}}.
    $
    Moreover, $\nu(\m A)<s-1$, because otherwise an $(s-1)$-matching in
    $\m A$, together with $\sigma$, would form an $s$-matching in $\FF$.

    Note that $\Sigma_{n,\overline{\sigma}}$ is isomorphic to $\m D_n$, and  the family $\m D_n$ is
    $\left(\frac n{12},\frac n4\right)$-spread by Lemma~\ref{dspread}. Thus, we may
    apply Lemma~\ref{lemc} to $\m A$. We note that $|\m Y| = (s-2)d_{n,1}$ in this case. Thus
    \[
        |\m A|
        \leq
        (s-2)d_{n,1}+\frac{1}{n^4}d_n.
    \]

    Put
    \[
        \m U_\sigma
        :=
        \bigcup_{i=1}^{s-1}
        \Sigma_{n,\overline{\sigma}}[(x_i,y_i)].
    \]
    Since the stars $\Sigma_n[(x_i,y_i)]$ are pairwise disjoint and
    $\sigma(x_i)\neq y_i$ for all $i$, the families
    $\Sigma_{n,\overline{\sigma}}[(x_i,y_i)]$ are also pairwise disjoint, and
    each has size $d_{n,1}$. Hence $
        |\m U_\sigma|=(s-1)d_{n,1}.
    $
    Therefore
    \[
        |\m U\setminus (\m F\setminus \m F')|\ge |\m U_\sigma\setminus \m A|
        \geq   d_{n,1}-\frac{1}{n^4}d_n.
    \]
    On the other hand, $\FF\setminus \m U\subset \FF'$, and so
    \[
        |\FF|
        \le 
        |(\FF\setminus \FF')\cap \m U|+|\FF'|
        \leq
        |\m U|
        -
        \left(d_{n,1}-\frac{d_n}{n^4}\right)
        +
        \frac{n!}{n^4} <|\m U|,
    \]
    a contradiction.   Thus
    $\FF=\m U$.

    Finally, pairwise disjointness of the stars means that, for every
    $i\neq j$, either        $x_i=x_j$ and $y_i\neq y_j$, or vice versa. 
\end{proofw}
{\bf Remark. } Let us discuss the proof of Theorem~\ref{emcder}. We run the same argument as in the proof of Theorem~\ref{emc}. The differences are mostly in calculations and the bounds on $r$ required for the application of the spread approximation machinery. The most significant difference is the definition of the families $\m A$ and $\m U_{\sigma}$. Since the family $\m F$ is the family of derangements, the families $\m A, \m U_{\sigma}$ become the family of double derangements (containing one of $(s-1)$ fixed elements). We proved bounds on the size  of the family of double derangements in Lemmas~\ref{permanent-min-degree},~\ref{ddspread}.  Thus we can similarly apply Lemma~\ref{lemc} to $\m A$ and get a non-trivial upper bound on its size. The rest of the proof is the same.

\section{Proof of Theorem~\ref{hm}} \label{sect_stab}

Let $\FF_1,\ldots,\FF_t$ be families in ${[n]\choose k}$. We say that
$\FF_1,\ldots,\FF_t$ contain a \emph{cross $t$-matching} if there exist pairwise
disjoint sets $F_1\in\FF_1,\ldots,F_t\in\FF_t$. We write
$f(n)\lesssim g(n)$ if $\limsup_{n\to\infty} f(n)/g(n)<1$.

To prove Theorem~\ref{hm}, we need the following lemma.

\begin{lemma} \label{cross-matching}
Let $n,t$ be positive integers satisfying $t\le n/(2^{17}\log n)$. Let
$\FF_1,\ldots,\FF_t$ be families of permutations such that
$\FF_i\subset \m D_n[(x_i,y_i)]$ for all $i$, where the pairs
$(x_i,y_i)\in[n]^2$ are distinct. If $\FF_1,\ldots,\FF_t$ do not contain a
cross $t$-matching, then either there exists $j\in[t]$ such that
\[
    \bigcup_{i=1}^t \FF_i
    \subset
    \bigcup_{i\in[t]\setminus\{j\}} \m D_n[(x_i,y_i)],
\]
or
\[
    \left|\bigcup_{i=1}^t \FF_i\right| \le(t-1.01)d_{n,1}.
\]
\end{lemma}

\begin{proof}
We may assume that $|\FF_1|\leq \cdots \leq |\FF_t|$. If $\FF_1=\emptyset$,
then the first alternative holds with $j=1$. Thus we may fix
$\sigma\in\FF_1$. We distinguish three cases.

\textbf{Case 1.} Suppose that $|\FF_1|\geq (n-1)!/10$.  Lemma~\ref{spreadsubfamily} with the ambient family being the family of all permutations implies that 
 each $\FF_i$ is  $n/40$-spread.

Put $$\m G_i:=\Big\{\sigma\setminus \{(x_i,y_i)\}: \sigma \in \FF_i, \sigma\cap \bigcup_{j\ne i}\{(x_j,y_j)\} = \emptyset\Big\}$$
We note that by $n/40$-spreadness of $\m F_i$ and our assumption on $t$,
$$|\m G_i|\ge \Big(1-\frac t{n/40}\Big)|\m F_i|\ge \frac 12|\m F_i|.$$
Applying Lemma~\ref{spreadsubfamily} again, we see that $\m G_i$ is $n/80$-spread. Now we can finish the proof as in Theorem~\ref{spreadapproxtheorem}. Consider the up-closures $\m H_i$ of $\m G_i.$  An  application of Theorem~\ref{spreadtheorem} with the same parameters $m,\delta$ gives that for each $i$ and a $1/8t$-random set $W$ we have $\Pr[W\in \m H_i]\ge \frac 12.$ Then we apply Lemma~\ref{lemkee} to conclude that there are pairwise disjoint $B_1\in \m G_1,\ldots, B_t\in \m G_t$. This, in turn, gives us a matching in the family $\m F$, a contradiction. 

\textbf{Case 2.} Suppose that $|\FF_1|<(n-1)!/10$ and
$|\FF_2|\geq (n-1)!/4$. If every $\pi\in\FF_1$ contains one of the pairs
$(x_i,y_i)$, $i\geq 2$, then the first alternative of the lemma holds. Hence we
may choose $\pi\in\FF_1$ such that $\pi(x_i)\neq y_i$ for all $i\geq 2$.

For $i\geq 2$, let $\m G_i$ be the subfamily of $\FF_i$ consisting of
permutations that do not intersect $\pi$. Note that $\m G_i = \m F_i\cap \m D_{n,\bar \pi}[(x_i,y_i)].$ We have a lower bound on the size of $\m D_{n,\bar \pi}[(x_i,y_i)]$ coming from Lemma~\ref{permanent-min-degree}. Thus, by the pigeon-hole principle, for $i\ge 2$ we have
$$|\m G_i|\ge |\m F_i|+|\m D_{n,\bar \pi}[(x_i,y_i)]|-|\m D_n[(x_i,y_i)]| >\Big(\frac 14+\frac 1{7.5}-\frac 1e\Big)(n-1)!-(n-2)!>0.01 (n-1)!,$$
provided $n>1000$ (we note that this inequality is satisfied by our assumption on $n,s$). Thus, $\m G_2,\ldots, \m G_t$ are $n/400$-spread. 
Applying the same spread-lemma argument as in
Case~1 to the families $\m G_i$, $i=2,\ldots,t$, we obtain pairwise disjoint 
$B_i\in\m G_i$, $i=2,\ldots,t$. Together with $\pi$, they form a cross
$t$-matching in $\FF_1,\ldots,\FF_t$, a contradiction.

\textbf{Case 3.} Finally, suppose that $|\FF_1|<(n-1)!/10$ and
$|\FF_2|<(n-1)!/4$. Then
\[
    \left|\bigcup_{i=1}^t \FF_i\right|
    <
    \frac{(n-1)!}{10}
    +
    \frac{(n-1)!}{4}
    +
    (t-2)d_{n,1}
    <
    (t-1.01)d_{n,1},
\]
where the last inequality holds, for instance, for $n\ge 100$ because $1/10+1/4+0.01=0.36<1/e$. 
\end{proof}

We are now ready to prove Theorem~\ref{hm}.

\begin{proofw}
\underline{\emph{Proof of Theorem~\ref{hm}}}.

The family $\Sigma_n$ is $(n/4,n/4)$-spread by Lemma~\ref{sspread}. We apply
Theorem~\ref{spreadapproxtheorem} to $\FF$ with
$q=\lfloor 2\log n\rfloor$. This gives a family $\m S$ and a subfamily
$\FF'\subset\FF$ such that $\FF\setminus\FF'\subset \Sigma_n[\m S]$ and
$|\FF'|\le n!/n^4$.

Suppose first that $\m S$ is non-trivial. By Corollary~\ref{corsimple}, applied
with $\varepsilon=1/2$, we have
$|\Sigma_n[\m S]|\le (s-3/2)(n-1)!$. Hence
\[
    |\FF|
    \leq
    |\Sigma_n[\m S]|+|\FF'|
    \leq
    \left(s-\frac32\right)(n-1)!+\frac{n!}{n^4}
    <
    (s-1)(n-1)!-d_{n,1}+1.
\]
Thus we may assume that $\m S$ is trivial.

Then there exist pairs $(x_1,y_1),\ldots,(x_{s-1},y_{s-1})\in[n]^2$ such that
$\Sigma_n[\m S]\subset \bigcup_{i=1}^{s-1}\Sigma_n[(x_i,y_i)]$. Since
$\tau(\FF)\geq s$, there exists $\sigma\in\FF$ such that
$\sigma(x_i)\neq y_i$ for all $i\in[s-1]$. In fact, this $\sigma$ must belong
to $\FF'$.

Let $\m A$ be the subfamily of $\FF\setminus\FF'$ consisting of permutations
that do not intersect $\sigma$. Then $\nu(\m A)<s-1$ and
$\m A\subset \Sigma_{n,\overline{\sigma}}$. Identifying
$\Sigma_{n,\overline{\sigma}}$ with $\m D_n$, we apply
Lemma~\ref{cross-matching} with $t=s-1$ to the families
$\FF_i:=\m A\cap\Sigma_n[(x_i,y_i)]$.

If the second alternative of Lemma~\ref{cross-matching} holds, then
$|\m A|\le (s-2.01)d_{n,1}$. Hence
\[
    |\FF|
    \leq
    |\m A|
    +
    \left|\bigcup_{i=1}^{s-1}\Sigma_{n,\sigma}[(x_i,y_i)]\right|
    +
    |\FF'|
    \le
    (s-1)(n-1)!-1.01 d_{n,1}+\frac{n!}{n^4}<(s-1)(n-1)!- d_{n,1},
\]
and the theorem follows. Therefore, we may assume that the first alternative
holds. Relabeling the pairs if necessary, we get
$\m A\subset \bigcup_{i=2}^{s-1}\Sigma_{n,\overline{\sigma}}[(x_i,y_i)]$.
Consequently,
\[
    \FF\setminus\FF'
    \subset
    \Sigma_{n,\sigma}[(x_1,y_1)]
    \cup
    \bigcup_{i=2}^{s-1}\Sigma_n[(x_i,y_i)].
\]

We claim that $\sigma$ is the only member of $\FF$ outside
$\bigcup_{i=1}^{s-1}\Sigma_n[(x_i,y_i)]$. Suppose, to the contrary, that there
is another such permutation, say $\pi$. Repeating the previous argument with
$\pi$ in place of $\sigma$, we find some $j\in[s-1]$ such that
\[
    \FF\setminus\FF'
    \subset
    \Sigma_{n,\pi}[(x_j,y_j)]
    \cup
    \bigcup_{\substack{i=1\\ i\neq j}}^{s-1}\Sigma_n[(x_i,y_i)].
\]

Assume first that \(j\neq 1\), say \(j=s-1\). Combining the two inclusions
above, we get
\[
    |\FF\setminus\FF'|
    \leq
    (s-1)(n-1)!-2d_{n,1}+(n-2)!.
\]
Since $|\FF'|\le n!/n^4$, this gives
$|\FF|< (s-1)(n-1)!-d_{n,1}$, and the theorem follows. Thus we may
assume that \(j=1\).
In this case,
\[
    \FF\setminus\FF'
    \subset
    \left(\Sigma_{n,\sigma}[(x_1,y_1)]
    \cap
    \Sigma_{n,\pi}[(x_1,y_1)]\right)
    \cup
    \bigcup_{i=2}^{s-1}\Sigma_n[(x_i,y_i)].
\]
Equivalently, inside the first star we lose all permutations that are disjoint
from at least one of $\sigma$ and $\pi$.

Since $\sigma\neq\pi$, there is $x'\neq x_1$ such that
$\sigma(x')\neq \pi(x')$. Consider permutations $\rho$ satisfying
$\rho(x_1)=y_1$, $\rho(x')=\sigma(x')$, and
$\rho(z)\neq \pi(z)$ for all $z\neq x_1,x'$. These permutations are disjoint
from $\pi$ but intersect $\sigma$, and their number is at least $d_{n-2}$.
Therefore the first star loses at least $d_{n,1}+d_{n-2}$ permutations, and so
\[
    |\FF|
    \leq
    (s-1)(n-1)!-d_{n,1}-d_{n-2}+\frac{n!}{n^4}
    <
    (s-1)(n-1)!-d_{n,1}+1.
\]
This contradiction proves that $\sigma$ is indeed the only member outside the $s-1$ stars.

Hence
\[
    \FF
    \subset
    \{\sigma\}
    \cup
    \Sigma_{n,\sigma}[(x_1,y_1)]
    \cup
    \bigcup_{i=2}^{s-1}\Sigma_n[(x_i,y_i)],
\]
where $\sigma(x_i)\neq y_i$ for all $i\in[s-1]$. It follows immediately that
$|\FF|\le (s-1)(n-1)!-d_{n,1}+1$.

Finally, equality can hold only if the stars are pairwise disjoint:
$\Sigma_{n,\sigma}[(x_i,y_i)]\cap
\Sigma_{n,\sigma}[(x_j,y_j)]=\emptyset$ for all distinct $i,j>1$, and
$\Sigma_{n,\sigma}[(x_1,y_1)]\cap \Sigma_n[(x_i,y_i)]=\emptyset$ for all
$i>1$. This is possible only if either $x_1=\cdots=x_{s-1}$ or
$y_1=\cdots=y_{s-1}$. The theorem follows.
\end{proofw}

\end{document}